\documentclass{article}
\usepackage[a4paper,margin=1in]{geometry}
\usepackage{xcolor}
\usepackage[colorlinks]{hyperref}
\definecolor{darkred}{rgb}{0.5,0,0}
\definecolor{darkblue}{rgb}{0,0,0.5}
\definecolor{darkgreen}{rgb}{0,0.5,0}
\hypersetup{colorlinks
  ,linkcolor=darkred
  ,filecolor=darkgreen
  ,urlcolor=darkred
  ,citecolor=darkblue,
  linktocpage}

\usepackage{hyperref}
\usepackage[ruled,vlined, linesnumbered]{algorithm2e}
\usepackage{tikz}

\usepackage{nicefrac}
\usepackage{amsmath, amssymb, amsfonts, mathtools}
\usepackage{amsthm, thm-restate}
\newtheorem{theorem}{Theorem}[section]

\newtheorem{lemma}[theorem]{Lemma}
\newtheorem{corollary}[theorem]{Corollary}

\newtheorem*{claim*}{Claim}
\newenvironment{claimproof}[1][Proof of the Claim.]{\proof[#1]}{\endproof}

\theoremstyle{remark}
\newtheorem{remark}{Remark}

\DeclareMathOperator{\tw}{tw}
\DeclareMathOperator{\tww}{tww}
\DeclareMathOperator{\stw}{stw}
\DeclareMathOperator{\rdeg}{red-deg}

\newcommand{\dotcup}{\mathbin{\dot\cup}}
\newcommand{\symdiff}{\mathbin{\triangle}}
\newcommand{\red}{{\text{red}}}

\tikzstyle{svertex}=[circle,inner sep=0.cm, minimum size=1.4mm, fill=black, draw=black]

\title{Twin-width of graphs with tree-structured decompositions}

\newcommand{\mailadress}[1]{\href{mailto:#1@mathematik.tu-darmstadt.de}{\texttt{#1}}}

\author{Irene Heinrich\footnote{The first author's research leading to these results has received funding from the European Research Council (ERC) under
the European Union’s Horizon 2020 research and innovation programme (EngageS: grant agreement No. 820148).} \quad and \quad Simon Raßmann\\[4mm] TU Darmstadt \\[4mm]
\(\mathclap{\{\mailadress{heinrich}, \mailadress{rassmann}\}\texttt{@mathematik.tu-darmstadt.de}}\)}

\begin{document}

\maketitle

\begin{abstract}
The twin-width of a graph measures its distance to co-graphs and generalizes classical width concepts such as tree-width or rank-width.
Since its introduction in 2020~\cite{tww1a},
a mass of new results has appeared relating twin-width to group theory, model theory, combinatorial optimization, and structural graph theory.

We take a detailed look at the interplay between the twin-width of a graph and the twin-width of its
components under tree-structured decompositions:
We prove that the twin-width of a graph of strong tree-width $k$ is at most $\frac{3}{2}k + o(k)$, contrasting nicely with the result of~\cite{twwexptwa}, which states that twin-width can be exponential
in tree-width.
Further, we employ the fundamental concept from structural graph theory of decomposing a graph into highly connected components,
in order to obtain optimal linear bounds on the twin-width of a graph given the widths of its biconnected components.
For triconnected components we obtain a linear upper bound if we add red edges to the components indicating the splits which led to the components.
Extending this approach to quasi-4-connectivity, we obtain a quadratic upper bound.
Finally, we investigate how the adhesion of a tree decomposition influences the twin-width of the decomposed graph.
\end{abstract}

\section{Introduction}
Twin-width is a new graph parameter, which has received significant attention in structural graph theory since its introduction in 2020~\cite{tww1a}.
The \emph{twin-width}\footnote{We refer to the preliminaries of this paper (subsections \emph{graphs and trigraphs} as well as \emph{twin-width}) for an equivalent definition of twin-width which is based on merging vertices instead of vertex subsets.} of a graph~$G$, denoted by $\tww(G)$, is the minimum width over all contraction sequences of~$G$ and a \emph{contraction sequence} of $G$ is roughly defined as follows: we start with a discrete partition of the vertex set of~$G$ into $n$ singletons where~$n$ is the order of~$G$.
Now we perform a sequence of $n-1$ merges, where in each step of the sequence precisely two parts are merged causing the partition to become coarser, until eventually, we end up with just one part -- the vertex set of $G$.
Two parts of a partition of $V(G)$ are \emph{homogeneously connected} if either all or none of the possible cross edges between the two parts are present in $G$.
The red degree of a part is the number of other parts to which it is not homogeneously connected.
Finally, the width of a contraction sequence is the maximum red degree amongst all parts of partitions arising when performing the sequence.

In~\cite{tww1a} the authors show that twin-width generalizes other width parameters such as rank-width, and, hence also clique-width and tree-width.
Furthermore, given a graph $H$, the class of $H$-minor free graphs has bounded twin-width
and FO-model checking is FPT on classes of bounded twin-width when contraction sequences of small width are provided as part of the input, see~\cite{tww1a}.
Many combinatorial problems which are NP-hard in general allow for improved algorithms if the twin-width of the input graph is bounded from above and the graph is given together with a contraction sequence of small width~\cite{tww3a}.

\subparagraph*{Motivation.}
To decompose a graph into smaller components and estimate a certain parameter of the original graph from the parameters of its components is an indispensable approach of structural graph theory.
Graph decomposition is crucial not only in taming branch-and-bound trees but also for theoretical considerations since it allows for stronger assumptions (i.e., high connectivity) on the considered graphs.
There are various ways to decompose a graph, e.g., into biconnected, triconnected, or quasi-4-connected components, into tree decompositions of small adhesion (the maximum cardinality of the intersection of adjacent bags), into modules, or decomposition into the factors of a graph product.

The twin-width of a graph in terms of its biconnected components has already been considered in~\cite{gonne},
where the author obtains a bound which, however, is not tight.
To date there is no detailed analysis of the relation between the twin-width of a graph and the twin-width of its triconnected, or quasi-4-connected components.
The only tight result towards ($k$-)connected components is the basic observation that the twin-width of a graph is the maximum over the twin-widths of its ($1$-connected) components.
While there already exists a strong analysis of the interplay of tree-width and twin-width (cf.~\cite{bounding-twwa}), it is still open how twin-width behaves with respect to the adhesion of a given tree decomposition, which can be significantly smaller than the tree-width of a graph (as an example, consider a graph whose biconnected components are large cliques -- the adhesion is~1 whereas the tree-width is the maximum clique size).
Further, there exist many variants of tree-width, for example, strong tree-width \cite{stwSeese, stwHalin} for which the interplay with twin-width remains unexplored in the existing literature.

\subparagraph*{Our results.}
We prove the following bound on the twin-width of a graph:
\begin{restatable}{theorem}{Stwvstww}\label{thm: tww-vs-stw}
	If $G$ is a graph of strong tree-width $k$, then
	\[\tww(G)\leq \frac{3}{2}k + 1 + \frac{1}{2}(\sqrt{k \ln k} + \sqrt{k} + 2\ln k).\]
\end{restatable}
This is a strong contrast to the result of~\cite{twwexptwa} that twin-width can be exponential in tree-width.
We further provide a class of graphs
for which the twin-width equals the strong tree-width,
showing that the linear dependence
of the above bound on \(k\) is necessary.
Further, we investigate how to bound the twin-width of a graph in terms of the twin-width of its highly connected components starting with biconnected components
and give a tight bound in this case.
\begin{restatable}{theorem}{Biconnectedbound} \label{Theorem: Bound on twin-width by twin-width of biconnected components}
	If \(G\) is a graph with biconnected components \(C_1,C_2\dots,C_\ell\),
	then
	\[\max_{i\in[\ell]}\tww(C_i)\leq \tww(G)\leq \max_{i\in[\ell]}\tww(C_i)+2.\]
\end{restatable}

Next, we consider decompositions into triconnected components:
\begin{restatable}{theorem}{Triconnectedtww}\label{Theorem: Bound on twin-width by twin-width of triconnected components}
	Let \(G\) be a \(2\)-connected graph and let $C_1, C_2, \dots, C_{\ell}$ be its triconnected components.
	For~$i \in [\ell]$ we construct a trigraph $\overline{C}_i$ from $C_i$ as follows:
	all virtual edges\footnote{That is, the pairs of vertices along which \(G\) was split to obtain the triconnected components.}
	of $C_i$ are colored red and all other edges remain black.
	If $C_i$ contains parallel edges, then we remove all but one of the parallel edges such
	that the remaining edge is red whenever one of the parallel edges was red.
	Then
	\[\tww(G)\leq \max\left(8\max_{i\in [\ell]}\tww(\overline{C}_i)+6,18\right).\]
\end{restatable}

Similarly clean decompositions into $k$-connected graphs with $k>3$ cannot exist~\cite{Gro16a}; but we move on one more step and consider the twin-width of a graph with respect to its quasi-4 connected components, introduced by~\cite{Gro16a}.
\begin{restatable}{theorem}{Quasifourtww}\label{Theorem: Bound on twin-width by twin-width of quasi-4-connected components}
	Let $C_1, C_2, \dots, C_{\ell}$ be the quasi-4-connected components of a 3-connected graph \(G\).
	\begin{enumerate}
		\item For $i \in [\ell]$ we construct a trigraph $\widehat{C}_i$ by adding for every 3-separator \(S\) in $C_i$
		along which~$G$ was split a vertex \(v_S\) which we connect via red edges to all vertices in \(S\).
		Then
		\[\tww(G)\leq \max\left(8\max_{i\in [\ell]}\tww(\widehat{C}_i)+14, 70\right).\]
		\item For $i \in [\ell]$,  we construct a trigraph $\overline{C}_i$
		by turning all 3-separators in $C_i$ along which~$G$ was split into red cliques.
		Then
		\[\tww(G)\leq \max\left(4\max_{i\in[\ell]}\left(\tww(\overline{C}_i)^2+\tww(\overline{C}_i)\right)+14,70\right).\]
	\end{enumerate}
\end{restatable}

For the general case of tree decompositions of bounded adhesion, we get the following:
\begin{restatable}{theorem}{BoundedAdhesiontww}\label{Theorem: Bound twin-width by twin-width of blocks, hat + torso version}
For every \(k \in \mathbb{N}\) there exist explicit constants \(D_k\) and \(D_k'\) such that
for every graph \(G\) with a tree decomposition of adhesion \(k\) and parts \(P_1,P_2,\dots,P_\ell\),
the following statements are satisfied:
\begin{enumerate}
\item For each \(P_i\), we construct a trigraph \(\widehat{P}_i\) by adding for each adhesion set \(S\) in \(P_i\)
	a new vertex \(v_S\) which we connect via red edges to all vertices in \(S\).
	Then
	\[\tww(G)\leq 2^k\max_{i\in[\ell]}\tww(\widehat{P}_i)+D_k.\]
\item Assume \(k\geq 3\). For each \(P_i\), we construct the torso \(\overline{P}_i\)
	by turning every adhesion set in \(P_i\) into a red clique. Then
	\[\tww(G)\leq \frac{2^k}{(k-1)!}\max_{i\in[\ell]} \tww(\overline{P}_i)^{k-1}+D_k'.\]
\end{enumerate}
\end{restatable}

In \cite{bounding-twwa} the authors gave an exponential bound for the twin-width of a graph in terms of its tree-width.
We refine their result by separating the dependence of this bound on the width and on the adhesion of the tree-decomposition.
Remarkably, while our bound is still exponential in the adhesion of the tree-decomposition, it is only linear in its width.
\begin{restatable}{theorem}{Twwvsadhesion} \label{thm:  widthadhesion}
	Let \(G\) be a graph with a tree decomposition of width~\(w\) and adhesion~\(k\).
	Then
	\[\tww(G)\leq 3\cdot 2^{k-1}+\max(w-k-1,0).\]
\end{restatable}

\subparagraph*{Bounding the red degree of decomposition trees.}
The underlying structure of all the decompositions that we consider in this paper is a tree.
We generalize the optimal contraction sequence for trees \cite{tww1a} which works as follows: choose a root for the tree. If possible, choose two sibling leaves and contract them (which implies a red edge from the new vertex to its parent). Whenever a parent is adjacent to two of its leaf-children via red edges, these two children are merged. This ensures that the red degree of any parent throughout the whole sequence never exceeds~2. 
If there are no sibling leaves, then choose a leaf of highest distance to the root and contract it with its parent. This yields a red edge between the new merged vertex and the former grandparent. Repeat this until we end up with a singleton.
We preserve this idea in our proofs to ensure that at no point in time three distinct bag-siblings contribute to the red degree of the vertices in their parent bag.

\subparagraph*{Further related work.}
A standard reference on tree-width is~\cite{Bod98}.
For the basics on graph connectivity and decomposition we refer text books on graph theory such as~\cite{West01}.
The twin-width of a graph given the twin-width of its modular decomposition factors
already has been investigated in~\cite{tww_polynomial_kernels, SAT_approach}.
In particular, the twin-width of a lexicographical product is just the maximal twin-width of its factors.
In contrast to the linear-time solvable tree-width decision problem~\cite{Bod96} (for a fixed $k$: is the tree-width of the input graph at most $k$?), deciding whether the twin-width of a graph is at most~4 is already NP-complete~\cite{twwleq4NPharda}.

\subparagraph*{Organization of the paper.}
We provide the preliminaries in Section~\ref{sec:prelims}.
Our results on strong tree-width can be found in Section~\ref{sec: stw}.
In Section~\ref{sec: small seps} we prove new bounds on the twin-width of a graph given the twin-widths of its highly connected components, and, we generalize our approach to graphs which allow for a tree-decomposition of small adhesion.

\section{Preliminaries}\label{sec:prelims}
For a natural number \(n\), we denote by \([n]\) the \(n\)-element set \(\{1,\dots,n\}\).
For a set \(A\), we write~\(\mathcal{P}(A)\) for the power set of \(A\).
For a natural number \(k\leq |A|\), we write \(\binom{A}{k}\) for the set of \(k\)-element subsets of \(A\).
 
\subparagraph*{Graphs and trigraphs.}
All graphs in this paper are finite, undirected and contain no loops.
For a graph \(G\), we denote its vertex set by \(V(G)\)
and its edge set by \(E(G)\). We write \(|G|\coloneqq|V(G)|\) for the order of \(G\).

A \emph{trigraph} is an edge-colored graph~\(G\) with disjoint sets
\(E(G)\) of \emph{black edges} and \(R(G)\) of \emph{red edges}. We can interpret
every graph as a trigraph by setting~\(R(G)=\emptyset\).
For a vertex subset $A$ of a trigraph~$G$, we denote by~\(G[A]\) the \emph{subgraph induced on \(A\)} and by \(G-A\) the subgraph
induced on \(V(G)\setminus A\). For a vertex~\(v\in V(G)\), we also write
\(G-v\) instead of \(G-\{v\}\).
If~\(G\) is a graph, the \emph{neighborhood}
and \emph{degree} of a vertex~\(v\in V(G)\)
is denoted by~\(N_G(v)\) and~\(d_G(v)\) (or \(N(v)\) and \(d(v)\) if \(G\) is clear from context) respectively.
For trigraphs, we write \(\rdeg_G(v)\) for the \emph{red degree} of~\(v\), i.e., the
degree of~\(v\) in the graph \((V(G),R(G))\).
We write \(\Delta(G)\) or \(\Delta_{\red}(G)\) for the maximum (red) degree of a (tri-)graph \(G\).
For a set of vertices \(A\subseteq V(G)\), we write \(N[A]\coloneqq A\cup\bigcup_{v\in A}N(v)\)
for the \emph{closed neighborhood of \(A\)}.

A \emph{multigraph} is a graph where we allow multiple edges between each pair of vertices.

\subparagraph*{Twin-width.}
Let \(G\) be a trigraph and \(x,y\in V(G)\) two distinct, not necessarily adjacent vertices of \(G\).
We \emph{contract \(x\) and \(y\)} by merging the two vertices to a common vertex \(z\),
leaving all edges not incident to~\(x\) or~\(y\) unchanged,
connecting \(z\) via a black edge to all common black neighbors of \(x\) and \(y\),
and via a red edge to all red neighbors of \(x\) or \(y\) and to all vertices
which are connected to precisely one of \(x\) and \(y\). We denote the resulting trigraph
by \(G/{xy}\).
A \emph{partial contraction sequence} of \(G\) is a sequence of trigraphs \((G_i)_{i\in[k]}\)
where \(G_1=G\) and~\(G_{i+1}\) can be obtained from~\(G_i\) by contracting two distinct vertices
\(x_i,y_i\in V(G_i)\). By abuse of notation, we also call the sequence \((x_iy_i)_{i<k}\) of contraction pairs a partial contraction sequence.
The width of a partial contraction sequence is the maximal red degree
of all graphs~\(G_1,\dots,G_k\). If the width of a sequence is at most \(d\), we call it a \emph{\(d\)-contraction sequence}.
A \emph{(complete) contraction sequence} is a partial contraction sequence whose final trigraph
is the singleton graph on one vertex.
The minimum width over all complete contraction sequences of~\(G\) is called the \emph{twin-width} of \(G\)
and is denoted by \(\tww(G)\).
We often identify a vertex~\(v\in V(G)\) with the vertices in the graphs \(G_i\) that~\(v\)
gets contracted to and sets of vertices with the sets of vertices they get contracted to.

Twin-width has many nice structural properties. For example, it is monotone with respect to induced subgraphs: for every induced subgraph \(H\subseteq G\) it holds that \(\tww(H)\leq\tww(G)\).
Moreover, the twin-width of a disconnected graph is the maximum twin-width of its connected components.

\subparagraph*{Tree decompositions and tree-width.}
Given a graph \(G\), a \emph{tree decomposition} of \(G\) is a pair \(\mathcal{T}=(T, \{B_i\colon i \in V(T)\})\) where \(T\) is a tree and
 \(\{B_i\colon i \in V(T)\}\) is a family of subsets of~\(V(G)\), one for each node of $T$, satisfying the following conditions:
\begin{enumerate}
	\item $V(G) = \bigcup_{i \in V(T)}B_i $,
	\item for every vertex \(v\) of $G$ the graph $T[\{i \in V(T)\colon v \in B_i\}]$
	is connected,
	\item for every edge $uv$ of $G$ there exists an $i \in V(T)$ with $\{u,v\} \subseteq B_i$.
\end{enumerate}
A set $B_i$ is a \emph{bag} of $\mathcal{T}$ and a subgraph of the form \(G[B_i]\) is a \emph{part} of $\mathcal{T}$.
The sets of the form \(B_i\cap B_j\) for edges \(ij\in E(T)\) are the \emph{adhesion sets} or \emph{separators} of \(\mathcal{T}\) and the maximal size of an adhesion set is the \emph{adhesion} of~\(\mathcal{T}\).
The graph obtained from a part~\(G[B_i]\) by completing all adhesion sets \(B_i\cap B_j\)
to cliques is the \emph{torso} of \(G[B_i]\).
The width of a tree-decomposition is \(\max_{i\in V(T)}|B_i|-1\)
and the minimum width over all tree decompositions of~\(G\) is the \emph{tree-width} of \(G\)
and is denoted by \(\tw(G)\).

\subparagraph*{Strong tree-width.}
\emph{Strong tree-width} or \emph{tree-partition width} is a graph parameter independently introduced by \cite{stwSeese} and \cite{stwHalin}.
A \emph{strong tree decomposition} of a graph~$G$ is a pair $(T, \{B_i\colon i \in V(T)\})$ where~$T$ is a tree and $\{B_i\colon i \in V(T)\}$ is a partition of~$V(G)$ with one part for each node of $T$ such that
for every edge $uv$ of $G$ there either exists an $i \in V(T)$ such that $\{u,v\} \subseteq B_i$ or there exist two adjacent nodes $i$ and $j$ in $T$ with~$u \in B_i$ and~$v \in B_j$.
The sets $B_i$ are called \emph{bags} and
$\max_{i \in V(T)}|B_i|$ is the \emph{width} of the decomposition.
The minimum width over all strong tree decompositions of~$G$ is the \emph{strong tree-width} \(\stw(G)\) of~$G$.

\begin{remark}[(Strong) tree-width and twin-width]
	The tree-width of a graph is bounded in its strong tree-width via \(\tw(G)\leq 2\stw(G)-1\), see \cite{stwSeese}.
	In the other direction, there is no bound: the strong-tree width of a graph is unbounded in its tree-width~\cite{bodlaender_domino_tw}.
	However, it holds that \(\stw(G)\in O(\Delta(G)\cdot\tw(G))\), see~\cite{tw_vs_stw}. Thus, for graphs of bounded degree, the two width
	notions are linearly equivalent.
	In general, the strong tree-width is unbounded in the twin-width of a graph.
	For example, consider a complete graph on $2n$ vertices.
	A width-minimal strong tree decomposition of this graph has two bags, each containing~$n$ vertices.
	However the twin-width of a complete graph is 0.
\end{remark}

\subparagraph*{Separators and highly connected components.}
A \emph{(\(k\)-)separator} of a graph \(G\) is a set \(S\subseteq V(G)\) of vertices
(with \(|S|=k\)) such that \(G-S\) contains more connected components
than \(G\). A \emph{cut vertex} is a vertex \(v\) for which \(\{v\}\) is a \(1\)-separator.
If \(G\) is connected, has order at least \(k+1\) and contains no \(\ell\)-separator for 
\(\ell<k\), then \(G\) is called \emph{\(k\)-connected}.

If \(S\) is a separator of \(G\), and \(C_1,\dots,C_\ell\) are the components
of \(G-S\) that are adjacent to \(S\), the graphs \(G[C_i\cup S]\)
are called the parts obtained by splitting \(G\) at \(S\).
If we complete \(S\) to a clique in these parts, we also call them
the torsos obtained by splitting \(G\) at \(S\).

If a graph is not \(k\)-connected,
we can split it along a small separator to obtain smaller torsos
with possibly higher connectivity.

If we split a graph \(G\) at all cut vertices,
the resulting parts are the maximal connected subgraphs of~$G$ that have no cut vertices,
and are called the \emph{biconnected components} of~$G$.
The \emph{block-cut-forest} of~\(G\) is a bipartite graph where one part is the set of biconnected components of $G$ and the other part is the set of cut vertices of $G$ and a biconnected component is adjacent to a cut vertex precisely if the vertex is contained in the component.
This graph is a forest, and even a tree if \(G\) is connected \cite{West01}.
In the latter case we also speak of the \emph{block-cut-tree} of $G$.
In terms of tree decomposition this can be rephrased as follows:
Every connected graph~\(G\) of order at least~2 has a unique tree decomposition~\(\mathcal{T}\)  of adhesion at most \(1\) such that every
part of~$\mathcal{T}$ is either \(2\)-connected or a complete graph of order \(2\).

When splitting along \(2\)-separators, we also get a tree-shaped decomposition (see~\cite{triconnected}):
For every \(2\)-connected graph, there exists a tree decomposition~\(\mathcal{T}\) of~\(G\) such that~\(\mathcal{T}\) has adhesion at most~\(2\), and the torso
of every bag is either \(3\)-connected, a cycle, or a complete graph of order \(2\). Moreover, the set of bags of this tree decomposition is isomorphism-invariant.

The \emph{triconnected components of \(G\)} are multigraphs constructed from the torsos of this tree decomposition.
In this work, these multigraphs are not important and we also call the torsos themselves \emph{triconnected components}.

A similarly clean decomposition into \(4\)-connected components arranged in a tree-like fashion does not exist \cite{Gro16a}.
This motivated Grohe to introduce the notion of quasi-4-connectivity \cite{Gro16a}: A graph \(G\) is \emph{quasi-4-connected} if it is \(3\)-connected
and all \(3\)-separators split off at most a single vertex. That is, for every separator \(S\) of size \(3\),
the graph \(G-S\) splits into exactly two connected components, at least one of which consists of a single vertex.
The prime example of quasi-4-connected graphs which are not 4-connected are hexagonal grids.
Also for quasi-4-connectivity there is a tree-like decomposition into components:
\begin{theorem}[{\cite{Gro16a}}]
	For every \(3\)-connected graph \(G\), there exists a tree decomposition~\(\mathcal{T}\) of~\(G\) such that \(\mathcal{T}\) has adhesion at most \(3\), and the torso
	of every bag is either quasi-4-connected or of size at most \(4\).
\end{theorem}
The torsos of this tree decomposition are called \emph{quasi-4-connected components of \(G\)}. In contrast to the case of connected, biconnected or triconnected components,
the quasi-4-connected components of a graph are not canonical,
meaning that there might exist multiple such tree decompositions with distinct
sets of bags. Nonetheless, there do exist canonical decompositions of \(3\)-connected graphs along small separators into torsos which are either quasi-4-connected or one of a few exceptional graphs~\cite{quasi4_canonical}. However, these parts do not come in the form
of a tree-decomposition, but in the form of a so-called \emph{mixed-tree-decomposition}~\cite{quasi4_canonical}.

\subparagraph*{Paley graphs.}
For a prime power \(q\) with \(q\equiv 1\pmod{4}\), the \emph{Paley graph} \(P(q)\)
is the graph whose vertex set is the \(q\)-element field \(\mathbb{F}_q\),
where two vertices \(x,y\in\mathbb{F}_q\) are adjacent in \(P(q)\)
if and only if~\(x-y\) is a non-zero square in \(\mathbb{F}_q\).
Since \(-1\) is a square in every finite field~\(\mathbb{F}_q\) with \(q\equiv 1\pmod 4\),
the Paley graphs are well-defined undirected graphs.
Paley graphs are strongly regular
with parameters \((q,\frac{q-1}{2},\frac{q-5}{4},\frac{q-1}{4})\), that is,
\(P(q)\) has order~\(q\), is~\(\frac{q-1}{2}\)-regular, every two adjacent vertices share exactly \(\frac{q-5}{4}\)
common neighbors, and every two non-adjacent vertices share exactly \(\frac{q-1}{4}\) common neighbors.
In~\cite{twwpaley}, the twin-width of the Paley graph \(P(q)\) was determined to be $\frac{q-1}{2}$.
The first four Paley graphs are depicted in Figure~\ref{fig:paley}.

\begin{figure}
	\centering
		\begin{tikzpicture}[scale=1.3]
		\def\krad{.9cm}
		\begin{scope}
			\def\vxnumber{5}
			\def\angle{360/5}
			\foreach \i in {1,...,\vxnumber}{
				\node[svertex] (\i) at (270-\angle/2+\i*\angle:\krad) {};
			}
			\draw (1)--(2)--(3)--(4)--(5)--(1);
		\end{scope}
		\begin{scope}[shift={(2.7,0)}]
			\def\vxnumber{9}
			\def\angle{360/9}
			\foreach \i in {1,...,\vxnumber}{
				\node[svertex] (\i) at (270-\angle/2+\i*\angle:\krad) {};
			}
			\draw (1)--(2)--(3)--(4)--(5)--(6)--(7)--(8)--(9)--(1)
			(1)--(3) (4)--(6) (7)--(9)
			(7)--(2)--(6)
			(9)--(5)--(1)
			(3)--(8)--(4)
			;
		\end{scope}
		\begin{scope}[shift={(5.4,0)}]
			\def\vxnumber{13}
			\def\angle{360/13}
			\foreach \i in {1,...,\vxnumber}{
				\node[svertex] (\i) at (270-\angle/2+\i*\angle:\krad) {};
			}
			\draw (1)--(2)--(3)--(4)--(5)--(6)--(7)--(8)--(9)--(10)--(11)--(12)--(13)--(1)
			(1)--(4)--(7)--(10)--(13)--(3)--(6)--(9)--(12)--(2)--(5)--(8)--(11)--(1)
			(1)--(5)--(9)--(13)--(4)--(8)--(12)--(3)--(7)--(11)--(2)--(6)--(10)--(1)
			;
		\end{scope}
		\begin{scope}[shift={(8.1,0)}]
			\def\vxnumber{17}
			\def\angle{360/17}
			\foreach \i in {1,...,\vxnumber}{
				\node[svertex] (\i) at (270-\angle/2+\i*\angle:\krad) {};
			}
			\draw (1)--(2)--(3)--(4)--(5)--(6)--(7)--(8)--(9)--(10)--(11)--(12)--(13)--(14)--(15)--(16)--(17)--(1)
			(1)--(3)--(5)--(7)--(9)--(11)--(13)--(15)--(17)--(2)--(4)--(6)--(8)--(10)--(12)--(14)--(16)--(1)
			(1)--(5)--(9)--(13)--(17)--(4)--(8)--(12)--(16)--(3)--(7)--(11)--(15)--(2)--(6)--(10)--(14)--(1)
			(1)--(9)--(17)--(8)--(16)--(7)--(15)--(6)--(14)--(5)--(13)--(4)--(12)--(3)--(11)--(2)--(10)--(1)
			;
		\end{scope}
	\end{tikzpicture}
\caption{The first four Paley graphs: $P(5), P(9), P(13)$, and $P(17)$}
\label{fig:paley}
\end{figure}

\section{Twin-width of graphs of bounded strong tree-width} \label{sec: stw}
In this section we are investigating the relationship between the twin-width and the strong tree-width of a graph.

For a graph $H$ and a vertex subset $U \subseteq V(H)$ a partial contraction sequence~$s$ of $H$ is a \emph{$U$-contraction sequence} if only vertices of $U$ are involved in the contractions in~$s$ and~$s$ is of maximal length with this property, that is, performing all contractions of $s$ yields a partition of $V(H)$ where $U$ forms one part and the rest of the parts are singletons.
We denote the minimum width over all $U$-contraction sequences of~$H$ by $\tww_U(H)$.

\Stwvstww*
\begin{proof}
	Let $\mathcal{T} = (T, \{B_i\colon i \in V(T)\})$ be a strong tree decomposition of $G$ of width $k$.
	Fix~$r \in V(T)$ and consider $T$ to be a rooted tree with root $r$ from now on.
	If a bag $B_i$ contains only one vertex $v$, then we set $v_{i} \coloneqq v$.
	A node $p$ of $T$ is a \emph{leaf-parent} if at least one of its children is a leaf.
	If~$B_i$ is a bag of $\mathcal{T}$, then \emph{contracting $B_i$} means to apply a width-minimal $B_i$-contraction sequence.
	After a contraction of two vertices~$u$ and~$v$ to a new vertex~$x$ we \emph{update} the strong tree decomposition~$\mathcal{T}$, that is,
	 if~$u$ and~$v$ were contained in the same bag, then we simply replace~$u$ and~$v$ by~$x$.
	If, otherwise, $u$ and $v$ were contained in different bags, then we remove~$u$ and~$v$ from their bags and insert~$x$ to one of the two bags that is closest to the root.
	If this causes an empty bag, then we remove the bag as well as the corresponding tree node.
	Observe that updating does not increase the width of any bag. In the following, we only perform
	contractions within a bag, between the only vertex in a leaf bag and a vertex in its parent bag,
	or between the only vertices in two sibling leaf bags.
	Hence, the updated strong tree decomposition always remains a valid strong tree decomposition of the contracted graph.
	\begin{algorithm*} 
		\caption{\textsc{Contract}($G$, $\mathcal{T}=(T, \{B_i, i \in V(T)\})$)}
		label every node of $T$ as unmerged\\
		\While{$|V(T)| \geq 2$}{
			\If{\(\mathcal{T}\) contains two merged sibling leaves \(\ell_1,\ell_2\),}
				{contract $v_{\ell_1}$ with $v_{\ell_2}$ (to $v_{\ell_1}$),\label{line: contract siblings}\\
				update $\mathcal{T}$, $\ell_1$ keeps the label merged}
			\ElseIf{\(\mathcal{T}\) contains an unmerged leaf~$\ell$,\label{if3}}
				{contract $B_{\ell}$,\label{line: contract leaf}  \\
				update \(\mathcal{T}\), label $\ell$ as merged}
			\ElseIf{\(\mathcal{T}\) contains a merged leaf \(\ell\) which is the only child of its parent \(p\),}
				{contract $B_p$,\label{line: contract parent}\\
				contract $v_p$ with $v_{\ell}$,\label{line: contract child to parent}\\
				update $\mathcal{T}$, label $p$ as merged}
			
		}
		Apply a width-minimal contraction sequence to the remaining graph.
		\label{alg: contract bd tww graph}
	\end{algorithm*}
	We claim that the algorithm \textsc{Contract} merges $G$ into a single vertex via a contraction sequence of the required width.
	
	First, we check that the algorithm terminates.
	Observe that  as long as the loop is executed, at least one of the if-conditions is satisfied in each iteration.
	Hence, $|V(G)|$ shrinks with every iteration, which proves that the algorithm terminates with a singleton graph, that is, it provides a contraction sequence.
	
	It remains to bound the width of the sequence.
	We start by observing that after every iteration of the loop, every bag is either an original bag from~\(\mathcal{T}\) or merged.
	Further, all merged bags are leaf bags in the updated tree decomposition. Finally, since in each iteration we contract at most
	one bag, and merged sibling leaf bags are contracted whenever possible in Line~\ref{line: contract siblings},
	every tree node \(p\) has at most two merged children at the end of each iteration, and at most one merged child
	at the start of Lines~\ref{line: contract leaf} and~\ref{line: contract parent}.
	
	For $a \in \mathbb{N}$ we set
	\[f(a) \coloneqq (a + \sqrt{a \ln a} + \sqrt{a} + 2\ln a)/ 2.\]
	We will exploit the result of~\cite{twwpaley} that an $a$-vertex graph has twin-width at most $f(a)$.

	Let $(G_i)_{i \leq |G|}$ be the contraction obtained by the algorithm.
	Fix $i \in [|G|]$ and $v \in V(G_i)$ and let $\mathcal{T}_i = (T_i, \mathcal{B}_i)$
	be the strong tree decomposition corresponding to $G_i$ and $B_j$ the bag containing $v$ in $\mathcal{T}_i$.
		
	If $j$ is neither a leaf, nor a leaf-parent, nor the parent of a leaf-parent in~$T_i$, then $\rdeg(v) = 0$.
		
	Assume that $j$ is a leaf of $T_i$, then all red edges incident to $v$ are either internal edges of $B_j$ or joining~$v$ with a vertex of $B_p$ where $p$ is the parent of $j$ in $T_i$.
	Since $\stw(G) \leq k$ there are at most $k$ red edges of the latter form.
	Internal red edges of a bag may only arise during the contraction of this leaf-bag in Line~\ref{line: contract leaf}.
	Since the corresponding partial contraction sequence is chosen to be width-minimal and by the bound of~\cite{twwpaley} we obtain that $\rdeg_{G_i}(v) \leq k + f(k)$.
	
	Now assume that $j$ is a leaf-parent in $T_i$.
	Then \(v\) is incident via red edges to vertices in at most two child bags of \(j\), one of which is merged.
	If the bag~$B_j$ of $\mathcal{T}_i$ was already contained in $\mathcal{T}$, then there are no internal red edges in $B_j$.
	Thus, the red degree of $v$ in $G_i$ is bounded by $k+1$.
	Otherwise $B_j$ is obtained during the contraction in Line~\ref{line: contract parent}.
	In this case, $j$ has precisely one child $\ell$ in $T_i$ and $\ell$ is merged.
	Thus,~\(v\) has red edges to this child, possibly to all \(k\) vertices of the parent bag of \(B_j\)
	and at most \(f(k)\) red neighbors within the bag \(B_j\).
	Hence, $\rdeg_{G_i}(v)\leq k+f(k)+1$.
		
	Finally,
	 assume that $j$ is neither a leaf nor a leaf-parent but parent of a leaf-parent in $T_i$.
	Then \(B_j\) has no internal red edges, and the only red edges incident to \(v\) are red edges to vertices in child bags of \(j\)
	which arise during the contractions in Lines~\ref{line: contract parent} and~\ref{line: contract child to parent}.
	Since each of the child bags of~\(j\) is contracted to one vertex before any contraction in another child bag happens,
	and \(j\) never has more than two partially contracted child bags, the red degree of $v$ is bounded by $k+1$.
\end{proof}
Theorem~\ref{thm: tww-vs-stw} in particular shows that the twin-width of a graph is bounded by a linear function
of its strong tree-width. In order to show that this linear dependence is necessary, we consider Paley graphs.
\begin{lemma} \label{lem:stwpaley}
	Let \(q\) be a prime power with \(q\equiv 1\pmod{4}\).
	The strong tree-width of the Paley graph \(P(q)\) is \(\frac{q-1}{2}\).
	\begin{proof}
	Recall that \(P(q)\) is strongly regular with parameters
	\((q,\frac{q-1}{2},\frac{q-5}{4},\frac{q-1}{4})\).
		
	We first provide a strong tree-decomposition of width \(\frac{q-1}{2}\) for \(P(q)\).
	The underlying tree of our decomposition is a path $T = xyz$.
	To define the bags, fix a vertex~\(v\in V(P(q))\) and choose as bags
	\(B_x\coloneqq\{v\}\), \(B_y\coloneqq N_{P(q)}(v)\), and \(B_z\coloneqq V(P(q))\setminus (B_x\cup B_y)\).
	Since \(P(q)\) is~\(\frac{q-1}{2}\)-regular and of order \(q\),
	all three bags have size at most \(\frac{q-1}{2}\). By construction, 
	there are no edges joining vertices of \(B_x\) and \(B_z\), which proves that $(T,\{B_x, B_y, B_z\})$ is indeed
	a strong tree decomposition of \(P(q)\).
	
	Conversely, let $\mathcal{T} = (T, \{B_i\colon i \in V(T)\})$ be a strong tree decomposition of~$G$.
	We show that \(\mathcal{T}\) has width at least \(\frac{q-1}{2}\). If \(\mathcal{T}\) has width
	greater than \(\frac{q-1}{2}\), we are done. Thus, assume from now on that \(\mathcal{T}\) has width
	at most \(\frac{q-1}{2}\). In particular, this implies that \(T\) has order at least \(3\),
	which implies that there exist two non-adjacent leaves in \(T\).
	
	Now, note that because \(1\) is a square in \(\mathbb{F}_q\), the graph \(P(q)\) has a Hamiltonian cycle
	and is hence \(2\)-connected. Thus, if some bag \(B_i\) of \(\mathcal{T}\) contains only a single vertex~\(v\),
	then \(i\) is a leaf of \(\mathcal{T}\).
	In this case, there is a unique neighbor $j$ of $i$ in $T$ and the \(\frac{q-1}{2}\) neighbors of \(v\) are contained in $B_j$,
	which proves that~\(\mathcal{T}\) has width \(\frac{q-1}{2}\).
	
	We are left with the case that every bag of \(\mathcal{T}\) contains at least two vertices
	and we show that no such strong tree decompositions can exist.
	For this, we note that by the parameters of the strongly regular graphs \(P(q)\), 
	for every pair of distinct vertices \(u,v\in V(P(q))\),
	the closed neighborhood \(N[\{u,v\}]\) contains either \(\frac{3}{4}(q-1)+1\)
	or \(\frac{3}{4}(q-1)+2\) many vertices, depending on whether \(u\) and \(v\) are adjacent or not.
	
	Now, consider a leaf \(\ell\) of \(T\) and denote the unique neighbor of \(\ell\) in \(T\) by \(p\).
	For every pair of distinct vertices \(u,v\in B_\ell\), we know that the whole closed neighborhood
	\(N[\{u,v\}]\) of \(u\) and \(v\) must be contained in \(B_\ell\cup B_p\).	
	Hence, \(|B_\ell\cup B_p|\geq \frac{3}{4}(q-1)+1.\)
	
	Since this observation is true for all leaf bags, all leaves of \(T\) must have the same neighbor in $T$.
	For another leaf bag \(B_{\ell'}\), we then get
	\begin{align*}
	|B_p|
	&=|B_\ell\cup B_p|+|B_{\ell'}\cup B_p|-|B_\ell\cup B_{\ell'}\cup B_p|\\
	&\geq\left(\frac{3}{4}(q-1)+1\right)+\left(\frac{3}{4}(q-1)+1\right)-q\\
	&=\frac{q+1}{2}.
	\end{align*}
	This contradicts our assumption that \(\mathcal{T}\) has width at most \(\frac{q-1}{2}\).
	\end{proof}
\end{lemma}

\begin{corollary}\label{cor: paley}
	Every Paley graph \(P(q)\) satisfies \(\tww(P(q))=\stw(P(q))\).
	In particular, for every $n \in \mathbb{N}$ there exists a graph $G_n$ of order at least $n$ which satisfies  \(\tww(G_n)=\stw(G_n)\).
\end{corollary}

\begin{proof}
	It is shown in~\cite{twwpaley} that $\tww(P(q)) = \frac{q-1}{2}$ for every prime power $q$ with $q \equiv 1 \pmod 4$. Combining this with Lemma~\ref{lem:stwpaley} yields the above statement.
\end{proof}

\section{Twin-width of graphs with small separators} \label{sec: small seps}
\subsection{Biconnected components}
We start our investigation of graphs of small adhesion by proving a bound on the twin-width of graphs in terms of the twin-width of their biconnected components.
This proof contains many of the ideas we will generalize later to deal with tri- and quasi-4-connected components as well as general graphs with a tree decomposition
of bounded adhesion.

The main obstacle to constructing contraction sequences of a graph from contraction sequences of its biconnected components
is that naively contracting one component might increase the red degree of the incident cut vertices in the neighboring components arbitrarily. Thus, we need to find contraction sequences of the biconnected components not involving the incident cut vertices.

Let~$G$ be a trigraph and~$\mathcal{P}$ be a partition of $V(G)$.
Denote by \(G/\mathcal{P}\) the trigraph obtained from~\(G\) by contracting each part of~\(\mathcal{P}\) into a single vertex. 
For a vertex $v \in V(G)$ we denote by~$\mathcal{P}(v)$ the part of~$\mathcal{P}$ that contains~$v$.
If $\mathcal{P}(v) \neq \{v\}$, then we obtain a refined partition~$\mathcal{P}_v$ by replacing~$\mathcal{P}(v)$ in~$\mathcal{P}$ by the two parts~$\mathcal{P}(v)\setminus\{v\}$ and~$\{v\}$.
Otherwise, we set~$\mathcal{P}_v = \mathcal{P}$.
Since~\(G/\mathcal{P}\)
can be obtained from~\(G/\mathcal{P}_v\) by at most one contraction,
and one contraction of a trigraph reduces the maximum red degree by at most \(1\) we have
\begin{equation} \label{eq: tww w r t one merge}
	\Delta_\red(G/\mathcal{P}_v)\leq \Delta_\red(G/\mathcal{P})+1.
\end{equation}

Recall that for a set \(U\subseteq V(G)\), we denote by \(\tww_U(G)\) the minimal width of a \(U\)-contraction sequence,
that is, a (partial) contraction sequence that only involves contractions within \(U\), that is of maximal length.
\begin{lemma}\label{lemma: contractions not using a vertex}
For every trigraph \(G\) and every vertex \(v\in V(G)\),
\[\tww_{V(G)\setminus\{v\}}(G)\leq \tww(G)+1.\]
\begin{proof}
Let~\((\mathcal{P}^{(i)})_{i\in[|G|]}\) be a sequence of partitions corresponding to a width-minimal contraction sequence of \(G\).
Further, let $j$ be the maximal index with~\( \{v\} \in \mathcal{P}^{(j)}\).
Then~\((\mathcal{P}^{(i)})_{i\in[j]}\) is a partial \(\tww(G)\)-contraction sequence which does not involve \(v\),
and by~\eqref{eq: tww w r t one merge} the sequence~\((\mathcal{P}^{(i)}_v)_{i\in[|G|]\setminus[j+1]}\) is a partial \((\tww(G)+1)\)-contraction
sequence which contracts the resulting trigraph until \(v\) and one further vertex remain.
Combining these two sequences yields the claim.
\end{proof}
\end{lemma}

\Biconnectedbound*

	\begin{proof}
		The lower bound follows from the fact that all biconnected components are induced subgraphs of~\(G\) together with the monotonicity of twin-width.
		
		For the upper bound we may assume that \(G\) is connected since the twin-width of a disconnected graph is the maximum twin-width of its connected components~\cite{tww1a}.
		Consider the block-cut-tree \(T\) of \(G\).
		The biconnected components and the cut vertices form a bipartition of~$T$.
		We choose a cut vertex \(r\) as a root of~$T$.
		For every biconnected component~\(C \in V(T)\), we let~\(v_C\) be the parent of~\(C\) in~\(T\).
		
		To make our argument simpler, let \(\widehat{G}\) be the graph obtained from \(G\)
		by joining a new vertex \(r_v\) to every vertex \(v \in V(G)\)  via a red edge.
		Similarly, for a biconnected component~\(C\), we let \(\widehat{C}\) be the graph obtained
		from \(C\) by attaching a new  vertex \(r_v\) to every vertex \(v\) of \(C\).
		We show that \(\tww(\widehat{G})\leq \max_{i\in[\ell]}\tww(C_i)+2\).
		The claim then follows since~\(G\) is an induced subgraph of~\(\widehat{G}\).
				
		\begin{claim*}
		For every biconnected component \(C\) of $G$,
		\[\tww_{V(\widehat{C})\setminus\{v_C\}}(\widehat{C})\leq \tww(C)+2.\]
		\begin{claimproof}
			By applying Lemma~\ref{lemma: contractions not using a vertex}
			to \(C\) and \(v_C\), we find a \((V(C)\setminus\{v_C\})\)-contraction sequence \(s\) of \(C\) of width at most \(\tww(C)+1\).
			We show how this contraction sequence can be adapted to also contract the vertices
			\(r_v\) for all \(v\in V(C)\).
			Indeed, before every contraction \(vw\) of \(s\), we insert the contraction of \(r_v\) and \(r_w\).
			This keeps the invariant that we never contract a vertex from \(C\) with a vertex \(r_v\),
			and further, every vertex of \(C\) is incident to at most one vertex \(r_v\) (or a contraction of those vertices).
			Moreover, the red degree among the vertices~\(r_v\) also stays bounded by \(2\).
			The entire partial contraction sequence constructed so far thus has width at most \(\tww(C)+2\).
			
			After applying this sequence, we end up with at most four vertices: \(v_C\), \(r_{v_C}\), the contraction of \(C-v_C\) and the
			contraction of all vertices \(r_v\) for vertices \(v\in V(C)\setminus\{v_C\}\). As \(r_{v_C}\) is only connected to \(v_C\)
			and the contraction of all other vertices \(r_v\) is not connected to \(v_C\), these four vertices form a path of length four.
			Thus, the contraction sequence can be completed with trigraphs of width at most~\(2\).
		\end{claimproof}
		\end{claim*}
		
		Now, consider again the whole graph \(\widehat{G}\) and choose a leaf block \(C\) of \(T\).
		We can apply the partial contraction sequence given by the previous claim to \(\hat{C}\)
		in \(\hat{G}\). Because we never contract \(v_C\) with any other vertex, this does not
		create red edges anywhere besides inside \(\hat{C}\). Thus, it is still a partial \((\tww(C)+2)\)-contraction sequence
		of \(\widehat{G}\). Moreover, the resulting trigraph is isomorphic to \(\widehat{G}-V(\widehat{C}-v_C-r_{v_C})\),
		that is, the graph obtained from \(\widehat{G}\) by removing the biconnected component \(C\) (but leaving the cut vertex \(v_C\) and its neighbor \(r_{v_C}\)).
		By iterating this, we can remove all biconnected components one after the other using width at most \(\max_{i\in[\ell]}\tww(C_i)+2\).
		Finally, we end up with precisely two vertices: The root cut vertex \(r\), together with its red neighbor \(r_r\), which we can simply contract.
	\end{proof}

Note that the bounds in Theorem~\ref{Theorem: Bound on twin-width by twin-width of biconnected components} are sharp even on the class of trees:
the biconnected components of a tree are its edges which have twin-width \(0\). As there are trees both of twin-width \(0\) and of twin-width~\(2\),
both the upper and the lower bound can be obtained.

\begin{corollary}
	Let \(\mathcal{C}\) be a graph class closed under taking biconnected components.
	Then \(\mathcal{C}\) has bounded twin-width if and only if the subclass of
	\(2\)-connected graphs in \(\mathcal{C}\) has.
	\begin{proof}
		Since every biconnected component is either \(2\)-connected or a complete graph on two vertices, we can bound the twin-width of every graph
		via the maximum twin-width of a \(2\)-connected biconnected component.
	\end{proof}
\end{corollary}

Moreover, Theorem~\ref{Theorem: Bound on twin-width by twin-width of biconnected components} also
reduces the algorithmic problem of computing good contraction sequences for a graph \(G\)
to the corresponding problem on the biconnected components of \(G\), which might be much smaller.

\begin{remark}
Note that in contrast to the case of connected components, where the twin-width of a graph is determined by the twin-widths of its components,
the twin-width of a connected graph is not determined by the multiset of its biconnected components.
As an example, consider a tree of twin-width \(2\) and a star with the same number of edges of twin-width \(0\).
Then both graphs have the same biconnected components, but different twin-width.
\end{remark}

\subsection{Apices and contractions respecting subsets}
To deal with adhesion sets of size at least \(2\), it no longer suffices to find contraction sequences of the parts that do not contract
vertices in the adhesion sets. Indeed, as those vertices can appear in parts corresponding to a subtree of unbounded depth,
this could create an unbounded number of red edges incident to vertices in adhesion sets. Instead, we want contraction sequences
that create no red edges incident to any vertices of adhesion sets.

For a trigraph~\(G\) and a set of vertices \(A\subseteq V(G)\) of red degree~\(0\),
we say that a partial sequence of \(d\)-contractions \(G=G_0,G_1,\dots,G_\ell\) \emph{respects \(A\)}
if \(G_i[A]=G[A]\) and \(\rdeg_{G_i}(a)=0\) for all~\(i\leq \ell\) and \(a\in A\).
Thus, for every contraction~\(xy\) in the sequence, we have \(x,y \notin A\)
and \(N(x)\cap A=N(y)\cap A\), which implies that the vertices in \(A\) are not incident
to any red edges all along the sequence.

A \emph{complete \(d\)-contraction sequence respecting~\(A\)} is a partial sequence of \(d\)-contractions
respecting~\(A\) of maximal length, i.e., one whose resulting trigraph~\(G_\ell\) does not allow a further contraction respecting~\(A\).
This is equivalent to no two vertices in \(V(G_\ell)\setminus A\) having the same neighborhood in~\(A\).
In particular, a complete contraction sequence respecting \(A\) leaves at most \(2^{|A|}\) vertices besides \(A\).

We write \(\tww(G,A)\) for the minimal \(d\) such that there exists a complete \(d\)-contraction sequence
respecting~\(A\). For a single vertex \(v\in V(G)\), we also write \(\tww(G,v)\) for \(\tww(G,\{v\})\).
Note that \(\tww(G)=\tww(G,\emptyset)\).

It was proven in \cite[Theorem 2]{tww1a} that adding a single apex to a graph of twin-width~\(d\) raises the twin-width to at most \(2d+2\).
The proof given there readily works in our setting without any modifications.
\begin{theorem}\label{theorem: Adding one apex keeps twin-width bounded}
	Let \(G\) be a trigraph, \(v\in V(G)\) a vertex not incident to any red edges and \(A\subseteq V(G)\setminus\{v\}\) a set of vertices. 
	Then
	\[\tww(G,A\cup\{v\})\leq 2\tww(G-v,A)+2.\]
\end{theorem}

\begin{corollary}\label{Corollary: Bound of pointed twin-width by twin-width}
Let \(G\) be a trigraph and \(A\subseteq V(G)\) a subset of vertices with \(\rdeg(a)=0\) for all vertices \(a\in A\).
Then
\[\tww(G,A)\leq 2^{|A|}\tww(G)+2^{|A|+1}-2.\]
\begin{proof}
We proceed by induction on \(|A|\).
For \(|A|=0\), the claim is immediate.
Thus, write \(A=A_0\dotcup\{a\}\) and assume the claim is true for \(A_0\).
Theorem~\ref{theorem: Adding one apex keeps twin-width bounded} and the induction hypothesis yield
\begin{align*}
\tww(G,A)
=\tww(G,A_0\cup\{a\})
&\leq 2\tww(G-a,A_0)+2\\
&\leq2\left(2^{|A_0|}\tww(G-A)+2^{|A_0|+1}-2\right)+2\\
&=2^{|A|}\tww(G-A)+2^{|A|+1}-2\\
&\leq2^{|A|}\tww(G)+2^{|A|+1}-2.\qedhere
\end{align*}
\end{proof}
\end{corollary}

\subsection{Tree decompositions of small adhesion}
We are now ready to generalize the linear bound on the twin-width of a graph in terms of its biconnected components
to allow for larger separators of bounded size. This is most easily expressed in terms of tree decompositions of bounded adhesion.

In all of the following two sections, let \(G\) be a graph, \(\mathcal{T}=((T,r), \{B_t\colon t \in V(T)\})\)
a rooted tree decomposition with adhesion~\(k\geq 1\).

For a node \(t\in V(T)\), we write \(P_t\coloneqq G[B_t]\) for the part associated to \(t\).
For a node~\(t\in V(T)\) with parent~\(s\in V(T)\) we write \(S_t\coloneqq B_t\cap B_s\)
and call \(S_t\) the \emph{parent separator} of \(P_t\)
or a \emph{child separator} of~\(P_s\).
Moreover, we set \(S_r\coloneqq\emptyset\) to be the \emph{root separator}.
For a node \(t\in V(T)\), we write~\(T_t\) for the subtree of~\(T\) with root~\(t\),
\(G_t\coloneqq G[\bigcup_{s\in V(T_t)} B_s]\) for the corresponding subgraph and~\(\mathcal{T}_t\) for the corresponding tree decomposition
of \(G_t\).

We may assume without loss of generality that every two nodes \(s,t\in V(T)\) with \(S_s=S_t\) are siblings.
Indeed, if they are not, let \(s'\) be a highest node in the tree with parent separator~\(S_s\)
and construct another tree decomposition by attaching all nodes \(t\) with \(S_t=S_s\) directly
to the parent of \(s'\) instead of their old parent. By repeating this procedure if necessary, we obtain the required property.
Note that this procedure does not affect the parts of \(\mathcal{T}\), nor which sets within the parts are the parent
and child separators.

For a node \(t\in V(T)\) with children \(c_1,\dots,c_\ell\), we set
\(N_{c_i}\coloneqq 
\{N(v)\cap S_{c_i}\colon v\in V(G_{c_i})\setminus S_{c_i}\}\)
to be the set of (possibly empty) neighborhoods
that vertices in \(G_{c_i}-S_{c_i}\) have in the separator~\(S_{c_i}\) (and thus in~\(P_t\)).
We now define a trigraph \(\widetilde{P}_t\) with vertex set
\[V(\widetilde{P}_t)\coloneqq V(P_t)\dotcup \{s^{c_i}_M\colon i\in[\ell], M\in N_{c_i}\}.\]
We will often abuse notation and also denote the set \(\{s^{c_i}_M\colon M\in N_{c_i}\}\)
by \(N_{c_i}\).

We define the edge set of \(\widetilde{P}_t\) such that
\begin{enumerate}
\item \(\widetilde{P}_t[V(P_t)]=P_t\),
\item \(\widetilde{P}_t[N_{c_i}]\) is a red clique for every \(i\),
\item \(s^{c_i}_M\) is connected via black edges to all vertices in \(M\),
\end{enumerate}
and there are no further red or black edges.
Note that in \(\widetilde{P}_t\), there are no red edges incident to any vertices in~\(P_t\) and thus in particular not to any vertices in \(S_t\).
A drawing of the gadget attached to \(S_{c_i}\) in~\(\widetilde{P}_t\) in comparison with the
simpler gadgets we will reduce to later can be seen in Figure~\ref{figure: gadgets attached to parts}.

\newcommand{\separator}{
	\node[shape=rectangle, fill=black, inner sep=3pt] (A1) at (0,0) {};
	\node[shape=rectangle, fill=black, inner sep=3pt] (A2) at (1,-0.5) {};
	\node[shape=rectangle, fill=black, inner sep=3pt] (A3) at (2,0) {};
	
	\draw (A1) -- +(-0.3,-0.2)
	      (A1) -- +( 0.1,-0.3)
	      (A2) -- +(-0.3, -0.2)
	      (A3) -- +(-0.3,-0.3)
	      (A3) -- +(-0.2,-0.4)
	      (A3) -- +( 0.2, -0.3);
}
\newcommand{\separatoredges}[1][dashed]{\draw[#1] (A1) -- (A2) -- (A3) -- (A1);}
\begin{figure}
\centering
\begin{tikzpicture}
\begin{scope}[shift={(0,0)}]
	\separator
	\separatoredges

	\node at (1,-1) {\(P_t\)};
\end{scope}

\begin{scope}[shift={(3.5,0)}]
	\separator
	\separatoredges
	
	\node[shape=circle, fill=black, inner sep=2pt] (A')    at (1,1.95) {};
	\node[shape=circle, fill=black, inner sep=2pt] (A'1)   at (-0.3,1.8) {};
	\node[shape=circle, fill=black, inner sep=2pt] (A'12)  at (0,1.35) {};
	\node[shape=circle, fill=black, inner sep=2pt] (A'2)   at (0.5,1) {};
	\node[shape=circle, fill=black, inner sep=2pt] (A'13)  at (1,0.9) {};
	\node[shape=circle, fill=black, inner sep=2pt] (A'123) at (1.5,1) {};
	\node[shape=circle, fill=black, inner sep=2pt] (A'23)  at (2,1.35) {};
	\node[shape=circle, fill=black, inner sep=2pt] (A'3)   at (2.3,1.8) {};
	
	\draw (A'1) -- (A1) -- (A'123)
	      (A'2) -- (A2) -- (A'123)
	      (A'3) -- (A3) -- (A'123);
	\draw (A1) -- (A'12) -- (A2) -- (A'23) -- (A3) -- (A'13) -- (A1);
	
	\draw[red] (A') -- (A'1) -- (A'2) -- (A'3) -- (A'12) -- (A'13) -- (A'23) -- (A'123) -- (A')
	           (A') -- (A'2) -- (A'12) -- (A'23) -- (A') (A'1) -- (A'3) -- (A'13) -- (A'123) -- (A'1)
	           (A') -- (A'3) -- (A'23) -- (A'1) -- (A'12) -- (A'123) -- (A'2) -- (A'13) -- (A');
	\draw[red] (A') -- (A'12) (A'1) -- (A'13) (A'2) -- (A'23) (A'3) -- (A'123);
	
	\node at (1,-1) {\(\widetilde{P}_t\)};
\end{scope}

\begin{scope}[shift={(7,0)}]
	\separator
	\separatoredges
	
	\node[shape=circle, fill=black, inner sep=2pt] (A) at (1, 1) {};
	\draw[red] (A1) -- (A) -- (A2) (A) -- (A3);
	
	\node at (1,-1) {\(\widehat{P}_t\)};
\end{scope}

\begin{scope}[shift={(10.5,0)}]
	\separator
	\separatoredges[red]
	
	\node at (1,-1) {\(\overline{P}_t\)};
\end{scope}
\end{tikzpicture}
\caption{A separator \(S_{t'}\) on three (square) vertices together with the three versions
	of gadgets we attach to it. Dashed edges represent either edges or non-edges.	
	In \(\widetilde{P}_t\), we add a red clique consisting of one vertex for every neighborhood of vertices in \(G_{t'}-S_{t'}\) in \(S_{t'}\).
	In \(\widehat{P}_t\), we only add a single
	vertex with red edges to all vertices in \(S_{t'}\). In \(\overline{P}_t\), we add no new vertices but turn
	all child separators into red cliques. }
\label{figure: gadgets attached to parts}
\end{figure}

\begin{lemma}\label{Lemma: Bound twin-width by twin-width of blocks}
	Let \(G\) and \(\mathcal{T}\) be as above.
	For every \(t\in V(T)\), it holds that
	\[\tww(G_t,S_t)\leq \max_{s\in V(T_t)}\tww(\widetilde{P}_s,S_s).\]
	In particular,
	\(\tww(G)\leq\max\limits_{s\in V(T)}\tww(\widetilde{P}_s,S_s)
	\leq 2^k\max\limits_{s\in V(T)}\tww(\widetilde{P}_s)+2^{k+1}-2\).

	\begin{proof}
		We proceed by induction on the height of \(t\) in the tree \(T\).
		If \(t\) is a leaf, then \(G_t=P_t=\widetilde{P}_t\) and we can use a contraction sequence of \(\widetilde{P}_t\) respecting \(S_t\).
		
		Thus, assume that \(t\) is not a leaf and assume the claim is true for all nodes of \(T\) below~\(t\).
		Let \(c_1,\dots,c_\ell\) be the children of \(t\) in \(T\). By the induction hypothesis,
		\[\tww(G_{c_i},S_{c_i})
		\leq \max_{s\in V(T_{c_i})}\tww(\widetilde{P}_s,S_s)
		\leq \max_{s\in V(T_t)}\tww(\widetilde{P}_s,S_s).\]
		Thus, we find complete contraction sequences of all \(G_{c_i}\) respecting \(S_{c_i}\) within the required width bounds.
		Because the graphs \(G_{c_i}\) and \(G_{c_j}\) for distinct \(i\) and \(j\) only intersect in \(S_{c_i}\cap S_{c_j}\),
		which are not contracted, we do not create red edges between the different graphs \(G_{c_i}\) nor between these graphs and
		their parent separators. Thus, we can combine these contraction sequences without exceeding our width bounds.
		This results in the graph \(\widetilde{P}_t\), where possibly some red edges within child separators of \(B_t\)
		are missing or black.
		We can thus in a final step apply a contraction sequence of \(\widetilde{P}_t\) respecting \(S_t\).
		
		The second claim follows from applying the first claim to \(t=r\)
		and then applying Corollary~\ref{Corollary: Bound of pointed twin-width by twin-width}. \qedhere
	\end{proof}
\end{lemma}

If \(\mathcal{T}\) has bounded width, we can proceed as in \cite[Lemma 3.1]{bounding-twwa} to bound the twin-width of the graphs~%
\(\widetilde{P}_t\) and thus the twin-width of \(G\): 
\begin{lemma}\label{Lemma: Bound on twin-width of graphs with small width and adhesion}
	Let \(G\) and \(\mathcal{T}\) be as above and additionally assume that \(\mathcal{T}\) has width at most~\(w\).
	For every \(t\in V(T)\), it holds that
	\[\tww(\widetilde{P}_t,S_t)\leq 3\cdot 2^{k-1}+\max(w-k-1,0).\]
	
	\begin{proof}
		We first note that the red degree of \(\widetilde{P}_t\) itself is bounded by \(2^k-1\).
		
		Now, let \(c_1,\dots,c_\ell\) be the (possibly empty) list of children of \(t\) in \(T\).
		We first find a contraction sequence of \(\bigcup_{i=1}^\ell N_{c_i}\)
		respecting \(S_t\). For this, we argue by induction that \(\bigcup_{i=1}^{j-1} N_{c_i}\)
		can be contracted while preserving the required width. This claim is trivial for \(j=1\).
		Hence assume we have already contracted \(\bigcup_{i=1}^{j-1}N_{c_i}\) to a set \(B_{j-1}\)
		of size at most \(2^{|S_t|}\). The vertices of \(B_{j-1}\) may be connected via red edges
		to vertices in \(B_{j-1}\) itself and to vertices in \(P_t\setminus S_t\). Thus, the red degree
		of vertices in \(B_{j-1}\) is bounded by
		\[|B_{j-1}|-1+|P_t|-|S_t|\leq 2^{|S_t|}+|P_t|-|S_t|-1\leq 2^k+w-k,\]
		where we used the fact that the function \(n\mapsto 2^n-n\) is non-decreasing on the natural numbers.
		The red degree of vertices in \(P_t\setminus S_t\) is bounded by
		\(|B_{j-1}|\leq 2^{|S_t|}\leq 2^k\).
		
		Now, we want to apply a maximal contraction sequence of \(N_{c_j}\) respecting \(S_t\).
		If \(S_{c_j}\subseteq S_t\), we do not need any contractions.
		Otherwise, pick an arbitrary vertex \(s\in S_{c_j}\setminus S_t\)
		and contract all pairs of vertices in \(N_{c_j}\) whose neighborhood differs only in \(s\).
		Because such a contraction will not create red edges to any vertex besides \(s\),
		and whenever this happens, the set \(N_{c_j}\) shrinks,
		the red degree of all vertices in \(N_{c_j}\) stays bounded by \(|N_{c_j}|-1\leq 2^k-1\) during this.
		But now we have contracted \(N_{c_j}\) to a set of size at most \(2^{k-1}\). If we now apply
		an arbitrary maximal contraction sequence of the remaining quotient of \(N_{c_j}\) respecting \(S_t\),
		we can get arbitrary red edges between \(N_{c_j}\) and \(S_{c_j}\). Still, the red degree of
		all vertices in \(N_{c_j}\) stays bounded by \(2^{k-1}-1+k\leq 2^k-1\).
		Let \(\bar{N}_{c_j}\) be the resulting quotient of \(N_{c_j}\).		
		In order to bound the red degree of vertices in \(P_t\setminus S_t\) during these contractions,
		we observe that every red neighbor of a vertex in \(P_t\setminus S_t\) in the quotient of \(N_{c_j}\)
		must be the contraction of at least two vertices. Thus, the red degree in \(P_t\setminus S_t\) is bounded by
		\[|B_{j-1}|+|N_{c_j}|/2\leq 2^k+2^{k-1}=3\cdot 2^{k-1}.\]
		
		Next, we contract vertices from \(B_{j-1}\) and \(\bar{N}_{c_j}\)
		which have equal neighborhoods in \(S_t\). As our bounds already allow
		every vertex in \(P_t\setminus S_t\) to be connected via red edges to all of \(B_{j-1}\cup\bar{N}_{c_j}\),
		it suffices to argue that this keeps the red degree of vertices in \(B_{j-1}\cup\bar{N}_{c_j}\),
		which has size at most \(2^{|S_t|}+2^{|S_t\cap S_{c_j}|}\leq 2^{|S_t|}+2^{k-1}\).
		Here we used that by our assumption, all nodes~\(t,t'\in V(T)\) with \(S_t=S_{t'}\) are siblings, which implies \(S_t\neq S_{c_j}\).		
		Hence, after one contraction, the red degree is bounded by
		\[|B_{j-1}|+|\bar{N}_{c_j}|+|P_t|-|S_t|-2\leq 3\cdot 2^{k-1}+w-k-1,\]
		because there are just not enough vertices left to break this bound.
		Here, we again used that \(2^{|S_t|}-|S_t|\leq 2^k-k\) by monotonicity.
		We have now successfully contracted \(N_{c_j}\) into \(B_{j-1}\) while keeping the red degree bounded by
		\[\max(2^k+w-k,2^k-1, 3\cdot 2^{k-1}, 3\cdot 2^{k-1}+w-k-1)=3\cdot 2^{k-1}+\max(w-k-1,0).\]		
		
		By repeating this procedure for all \(j\in[\ell]\), we find a contraction sequence
		of \(\bigcup_{i=1}^\ell N_{c_i}\) respecting \(S_t\) within this width.
		The resulting graph thus consists of \(S_t\), the vertices of~\(P_t\setminus S_t\)
		and the vertices from \(B_\ell\). In total, these are at most
		\(2^{|S_t|}+|P_t|-|S_t|\leq 2^k+w+1-k\) vertices besides those in \(S_t\).
		These can further be contracted while keeping the red degree bounded by
		\[2^k+w-k-1\leq 3\cdot 2^{k-1}+w-k-1.\]
		In total, our contraction sequence thus has width at most
		\(3\cdot 2^{k-1}+\max(w-k-1,0),\)
		proving the claim.
	\end{proof}
\end{lemma}

By combining Lemma~\ref{Lemma: Bound on twin-width of graphs with small width and adhesion} with
Lemma~\ref{Lemma: Bound twin-width by twin-width of blocks},
we obtain a general bound on the twin-width of graphs admitting a tree decomposition of bounded
width and adhesion: 
\Twwvsadhesion*

This upper bounds sharpens the bound given in \cite{bounding-twwa} by making explicit the dependence on the adhesion of the tree decomposition.
Our bound shows that, while the twin-width in general can be exponential in the tree-width \cite{twwexptwa}, the exponential dependence comes from the adhesion of the tree decomposition
and not from the width itself.

Moreover, up to multiplicative factors, our upper bound matches the known lower bounds.
As already mentioned, it is known that there are graphs whose twin-width is exponential in the adhesion of some tree decomposition~\cite{twwexptwa}. Adding a Paley graph into some bag whose twin-width is linear in its order shows that also the linear width term cannot be significantly improved upon.

\subsection{Simplifying the parts}
Before we apply this general lemma to the special case of the tree of tri-
or quasi-4-connected components, we show that we can simplify the gadgets attached
in the graphs \(\widetilde{P}\) to all separators while raising the twin-width by
at most a constant factor.

In a first step, we replace the sets \(N_{c_i}\) from the definition of the parts \(\widetilde{P}_t\)
by a single common red neighbor for every separator.
For every node \(t\in V(T)\) with children \(c_1,\dots,c_\ell\),
we define the trigraph~\(\widehat{P}_t\) as follows:
we set \(\mathcal{S}(t)\coloneqq\{S_{c_i}\colon i\in[\ell], S_{c_i}\not\subsetneq S_{c_j} \text{ for all } j\in[\ell]\}\)
to be the set of subset-maximal child separators of \(P_t\).
Now, we take a collection of fresh vertices \(V_{\mathcal{S}}\coloneqq\{v_S\colon S\in\mathcal{S}(t)\}\)
and set \[V(\widehat{P}_t)\coloneqq V(P_t)\dotcup V_{\mathcal{S}}.\]
The subgraph induced by \(\widehat{P}_t\) on \(V(P_t)\) is \(P_t\) itself.
The vertex \(v_S\) is connected via red edges to all vertices in \(S\) and has no further neighbors.
A drawing of the gadget attached to~\(S_{c_i}\) in \(\widehat{P}_t\) can be found in Figure~\ref{figure: gadgets attached to parts}.

\begin{lemma}\label{Lemma: Bound twin-width by twin-width of blocks, simpler gadgets} 
	Let \(G\) and \(\mathcal{T}\) be as before. Then for every \(t\in V(T)\), it holds that
	\[\tww(\widetilde{P}_t,S_t)\leq \max(2^k\tww(\widehat{P}_t)+2^{k+1}-2, 4^k+2^k-2).\]
	In particular,
	\(\tww(G)\leq \max(2^k\max_{t\in V(T)}\tww(\widehat{P}_t)+2^{k+1}-2,4^k+2^k-2).\)

	\begin{proof}
		Consider a separator~\(S\in\mathcal{S}(t)\) in the graph~\(\widetilde{P}_t\) and assume
		~\(t\) has more than one child~\(c\) with \(S_c\subseteq S\). If \(c\) and \(c'\) are such
		children, we can contract vertices from \(N_c\) and \(N_{c'}\) with the same
		neighborhood in \(S\) as long as these exist. By doing this for all children whose parent
		separator is contained in \(S\), we can reduce to the case that there is only a single such child
		using a contraction sequence of width at most \(2^{k+1}-2\).
		Thus, in the following assume that \(S_c\not\subseteq S_{c'}\) for every two distinct children
		\(c\) and \(c'\) of \(t\).
		
		By Corollary~\ref{Corollary: Bound of pointed twin-width by twin-width},
		it holds that \(\tww(\widetilde{P}_t,S_t)\leq 2^k\tww(\widetilde{P}_t)+2^{k+1}-2\).
		Thus, we want to bound \(\tww(\widetilde{P}_t)\).		
		For this, let \(c_1,\dots,c_\ell\) be the children of \(t\)
		and choose a vertex \(x_i\in N_{c_i}\) for every~\(i\in[\ell]\).
		We show that we can contract all vertices of \(N_{c_i}\) into \(x_i\) one after the other
		in such a way that the red degree of~\(x_i\) stays bounded by \(2^k-1\) while
		we do not create red edges between vertices of \(P_t\) and vertices of~\(N_{c_i}\) besides \(x_i\).
		
		Indeed, let \(S_{c_i}=\{s_1,\dots,s_{k'}\}\). If \(k'<k\), then we can contract the vertices of \(N_{c_i}\)
		into~\(x_i\) in an arbitrary way while keeping the red degree bounded by \(2^{k'}+k'\leq 2^k-1\).
		Hence, assume~\(k'=k\). Now, let \(M\coloneqq N(x_i)\cap S_{c_i}\) be the neighborhood of~\(x_i\)
		in the separator~\(S_{c_i}\) and define \(M_j\coloneqq M\symdiff\{s_1,\dots,s_j\}\).
		For all \(j\in[k]\), there exists at most one vertex \(y_j\coloneqq s^{c_i}_{M_j}\in N_{c_i}\) such that
		\(N(y_j)\cap S_{c_i}=M_j\). We first show that
		we can contract the vertices \(y_1,\dots,y_k\) into \(x_i\) one after the other
		without exceeding our bound of \(2^k-1\). For this, we note that
		after having contracted \(y_1,\dots,y_j\) into~\(x_i\),
		the remaining red neighbors of \(x_i\) are among
		\(\{s_1,\dots,s_j\}\cup (N_{c_i}\setminus\{x_i,y_1,\dots,y_j\})\).
		Now, whether or not \(y_j\) exists, the exclusion of~\(y_j\) removes
		one of the possible \(2^{|S_{c_i}|}\) many vertices in \(N_{c_i}\).
		Thus,~\(x_i\) has at most~\((2^k-1)\)-many red neighbors.
		
		After having merged \(y_1,\dots,y_k\) into \(x_i\), this possible neighborhood
		includes all the vertices in \(S_{c_i}\). Thus, we can now contract
		the remaining vertices of \(N_{c_i}\) in an arbitrary way without exceeding
		the bound of \(2^k-1\) on the red degree of \(x_i\).
		
		By applying this procedure for all \(i\in[\ell]\), we find a contraction sequence
		which contracts~\(\widetilde{P}_t\) to \(\widehat{P}_t\) using width at most \(2^k-1\).
		Thus, we have
		\[\tww(\widetilde{P}_t)\leq\max(2^k-1,\tww(\widehat{P}_t)).\]
		
		Using Corollary~\ref{Corollary: Bound of pointed twin-width by twin-width}, we thus have
		\begin{align*}
		\tww(\widetilde{P}_t,S_t)
		&\leq 2^k\tww(\widetilde{P}_t)+2^{k+1}-2\\
		&\leq 2^k\max(2^k-1,\tww(\widehat{P}_t))+2^{k+1}-2\\
		&=    \max\left(4^k+2^k-2, 2^k\tww(\widehat{P}_t)+2^{k+1}-2\right).
		\end{align*}
		
		The second claim follows from Lemma~\ref{Lemma: Bound twin-width by twin-width of blocks}
		together with the observation that \(\tww(G)=\tww(G,S_r)\)
		for the root separator \(S_r=\emptyset\) of \(\mathcal{T}\).
	\end{proof}
\end{lemma}

Next, we want to define a version \(\overline{P}_t\) of the parts which does not need extra vertices in~\(P_t\)
but instead marks the separators via red cliques. Indeed, let \(\overline{P}_t\) be the trigraph obtained from
\(P_t\) by turning each of the sets \(S\in\mathcal{S}_t\) into a red clique.
Thus, apart from possibly some missing edges within the parent separator \(S_t\),
the underlying graphs of the trigraphs \(\overline{P}_t\) are the \emph{torsos} of the tree decomposition.
We thus call the graphs \(\overline{P}_t\) the \emph{red torsos} of the tree decomposition \(\mathcal{T}\).

In order to obtain a bound on the twin-width of \(G\) in terms of the twin-width of the red torsos \(\overline{P}_t\),
we need one combinatorial lemma, which is a variant of Sperner's theorem \cite{LYM}.
\begin{lemma}\label{lemma: Sperner's theorem}
	Let \(\mathcal{F}\subseteq\mathcal{P}([n])\) be a family of subsets of \([n]\)
	such that every set has size at most \(k\)
	and no set is contained in another.
	\begin{enumerate}
		\item If \(k\leq\lfloor n/2\rfloor\), then \(|\mathcal{F}|\leq \binom{n}{k}\).
		\item If \(k>\lfloor n/2\rfloor\), then \(|\mathcal{F}|\leq \binom{n}{\lfloor n/2\rfloor}\leq\binom{2k-1}{k}\).
	\end{enumerate}
	\begin{proof}
	By the LYM-inequality \cite{LYM}, it holds that
	\[\sum_{A\in\mathcal{F}}\frac{1}{\binom{n}{|A|}}\leq 1.\]
	Furthermore, we also know that
	\[\sum_{A\in\mathcal{F}}\frac{1}{\binom{n}{|A|}}
	\geq|\mathcal{F}|\cdot \frac{1}{\max_{i\leq k}\binom{n}{i}}
	=\begin{cases}
	\frac{|\mathcal{F}|}{\binom{n}{k}} &\text{if } k\leq\lfloor n/2\rfloor,\\
	\frac{|\mathcal{F}|}{\binom{n}{\lfloor n/2\rfloor}} &\text{if } k>\lfloor n/2\rfloor.\\
	\end{cases}
	\]
	Combining these two inequalities yields everything but the very last inequality of the claim.
	This inequality follows from the fact that when \(k>\lfloor n/2\rfloor\), then we also have \(k>n/2\) and thus \(n\leq 2k-1\).
	\end{proof}
\end{lemma}

\begin{corollary}\label{corollary: bound red degree in hypergraph by red degree in torso}
Let \(H\) be a graph, \(\mathcal{F}\subseteq\mathcal{P}(V(H))\) and \(k\geq 2\) such that
\begin{enumerate}
\item every set in \(\mathcal{F}\) contains at most \(k\) vertices,
\item no two sets in \(\mathcal{F}\) are contained in each other, and
\item for each \(A\in\mathcal{F}\), the graph \(H[A]\) is a clique.
\end{enumerate}
Then every vertex \(x\in V(H)\) is contained in at most
\(\max\left(\binom{\Delta(H)}{k-1},\binom{2k-3}{k-1}\right)\)
sets of \(\mathcal{F}\).
\begin{proof}
For a vertex \(x\in V(H)\),
consider the family \(\mathcal{F}_{-x}\coloneqq\{A\setminus\{x\}\colon A\in \mathcal{F}, x\in A\}\),
whose cardinality is the number we want to bound.
All sets in \(\mathcal{F}_{-x}\) are subsets of \(N_H(x)\), have size at most \(k-1\) and do not contain
any other set in \(\mathcal{F}_{-x}\).
Thus, Lemma~\ref{lemma: Sperner's theorem} yields that
\[|\mathcal{F}_{-x}|\leq\max\left(\binom{|N_H(x)|}{k-1},\binom{2k-3}{k-1}\right)\leq \max\left(\binom{\Delta(H)}{k-1},\binom{2k-3}{k-1}\right).\qedhere\]
\end{proof}
\end{corollary}

\begin{lemma}\label{Lemma: Bound twin-width by twin-width of blocks, torso version}
	If \(k\geq 2\), then for every \(t\in V(T)\), it holds that
	\[\tww(\widehat{P}_t)\leq \max\left(k+1,\tww(\overline{P}_t)+\binom{\tww(\overline{P}_t)}{k-1},\tww(\overline{P}_t)+\binom{2k-3}{k-1}\right).\]
	In particular,
	\begin{align*}
		\tww(G)\leq
		\max\left(
		\begin{array}{l}
			2^k\max_{t\in V(T)}\left(\tww(\overline{P}_t)+\binom{\tww(\overline{P}_t)}{k-1}\right)+2^{k+1}-2,\\
			2^k\max_{t\in V(T)}\tww(\overline{P}_t)+2^k\binom{2k-3}{k-1}+2^{k+1}-2,\\
			4^k+2^k-2
		\end{array}
		\right)
	\end{align*}

	\begin{proof}
		Let \((x_iy_i)_{i<|\overline{P}_t|}\) be a minimal contraction sequence of \(\overline{P}_t\),
		which we can also interpret as a contraction sequence of \(P_t\).
		Recall that in the graph \(\widehat{P}_t\), we marked the child separators of \(P_t\) by
		adding extra vertices \(v_S\) for every subset-maximal child separator \(S\) which we connected
		to all vertices in \(S\) via red edges. As before, we denote the set of these
		extra vertices by \(V_{\mathcal{S}}\) and want to construct
		a contraction sequence for \(\widehat{P}_t\) from the sequence \((x_iy_i)_{i<|\overline{P}_t|}\)
		by interleaving this sequence with appropriate contractions within \(V_{\mathcal{S}}\).
		We will not contract vertices from \(V(P_t)\) with vertices in \(V_{\mathcal{S}}\)
		until the very last contraction. This means that it makes sense to talk about
		vertices in the quotient of \(V_{\mathcal{S}}\), that is, contractions of
		vertices in \(V_\mathcal{S}\).
		The contraction sequence can be found in Algorithm~\ref{alg: contract hat(P)_t}.
		\begin{algorithm*} 
		\For{\(i<|\overline{P}_t|\)}{
			\While{there exist distinct vertices \(v_S\) and \(v_{S'}\) in the quotient of \(V_{\mathcal{S}}\)
				such that \(N(v_S)/{x_iy_i}\subseteq N(v_{S'})/{x_iy_i}\)	\label{line: beginning of for-loop}}{
				contract \(v_S\) and \(v_{S'}\)
			}
			contract \(x_i\) and \(y_i\)
		}
		contract the remaining two vertices
		\caption{\textsc{Contract}($\widehat{P}_t$, \((x_iy_i)_{i<|\overline{P}_t|})\)}
		\label{alg: contract hat(P)_t}
		\end{algorithm*}
		
		To see that this contraction sequence does indeed contract \(\widehat{P}_t\) to a single vertex,
		note that after having exited the for-loop, we have applied all contractions \(x_iy_i\),
		which means that~\(P_t\) was contracted to a single vertex. But then each two vertices \(v_S\)
		and \(v_{S'}\) in the quotient of~\(V_{\mathcal{S}}\) have identical neighborhood in \(V(P_t)\),
		which means that they were contracted before.
		
		Now, in order to bound the width of this sequence, we first argue that we preserve the loop-invariant
		that at the beginning of Line~\ref{line: beginning of for-loop}, no two vertices \(v_S\) and \(v_{S'}\)
		have neighborhoods that contain each other. In the uncontracted graph \(\widehat{P}_t\), this is true
		by construction, as we only added a vertex \(v_S\) for every subset-maximal child separator \(S\).
		If the invariant is true at the start of the \(i\)-th iteration, then the process in the while-loop
		precisely ensures that the property is also true at the start of the \((i+1)\)-th iteration.
		
		Next, we argue that the red-degree of all vertices \(v_S\) is bounded by \(k\) at the beginning of each iteration
		and bounded by \(k+1\) during the whole sequence. Moreover, at the start of each iteration, the (red) neighborhood
		of each vertex \(v_S\) forms a red clique in \(\overline{P}_t\).
		Again, this is true by construction in the uncontracted graph \(\widehat{P}_t\).
		If this invariant is true at the beginning of the \(i\)-th iteration, we show that it is also true at the beginning of the \((i+1)\)-th iteration.
		For this, we note that whenever we contract a vertex \(v_S\) into a vertex \(v_{S'}\) because \(N(v_S)/{x_iy_i}\subseteq N(v_{S'})/{x_iy_i}\),
		then both neighborhoods contain either \(x_i\) or \(y_i\) (or both) and both \(x_i\) and \(y_i\) are contained in one
		of the two neighborhoods.
		Thus, the neighborhood of the contracted vertex contains both~\(x_i\) and~\(y_i\) and its size is bounded by
		\(|N(v_{S'})\cup\{x_i,y_i\}|\leq k+1\). Moreover, all edges in \(\overline{P}_t[N(v_{S'})\cup\{x_i,y_i\}]\)
		besides possibly \(x_iy_i\) are red, which means that after contracting \(x_i\) and \(y_i\),
		the neighborhood of the contraction of \(v_S\) and \(v_{S'}\) is again a red clique in \(\overline{P}_t\).
		
		It remains to bound the red degree of the vertices in (the quotients of) \(P_t\).
		For this, we note that during the while-loop of Algorithm~\ref{alg: contract hat(P)_t},
		the red degree of vertices in \(P_t\) can only decrease. Thus, it suffices to bound the red degree
		at the start of each iteration. The red neighbors of vertices in \(P_t\) come in two sorts:
		red neighbors in \(P_t\) itself, stemming from the contraction sequence \((x_iy_i)_{i<|\overline{P}_t|}\),
		and red neighbors among the vertices in \(V_{\mathcal{S}}\).
		
		The red degree that a vertex \(x\) in \(P_t\) can obtain within \(P_t\) is bounded by the width of the
		sequence \((x_iy_i)_{i<|\overline{P}_t|}\) which is \(\tww(\overline{P}_t)\).
		To bound the red degree that a vertex \(x\) in \(P_t\) can get from vertices in \(V_{\mathcal{S}}\),
		let \(Q_i\) be the partially contracted graph that we have at the start of the \(i\)-th iteration
		and set \(\mathcal{F}_i\coloneqq\{N_{Q_i}(v_S)\colon v_S\in V_{\mathcal{S}}\}\).
		By Corollary~\ref{corollary: bound red degree in hypergraph by red degree in torso}, the number
		of red neighbors of \(x\) among the vertices \(v_S\) is bounded by
		\[  \max\left(\binom{\Delta_{\red}(Q_i)}{k-1},\binom{2k-3}{k-1}\right)
		\leq\max\left(\binom{\tww(\overline{P}_t)}{k-1},\binom{2k-3}{k-1}\right).\]

		Combining these bounds, we get that the red degree of the contraction sequence of \(\widehat{P}_t\) is bounded by
		\[\max\left(k+1,\tww(\overline{P}_t)+\binom{\tww(\overline{P}_t)}{k-1}, \tww(\overline{P}_t)+\binom{2k-3}{k-1}\right).\]
				
		The second claim follows by inserting this bound into the bound in Lemma~\ref{Lemma: Bound twin-width by twin-width of blocks, simpler gadgets}.
	\end{proof}
\end{lemma}

Combining Lemma~\ref{Lemma: Bound twin-width by twin-width of blocks, simpler gadgets}
and Lemma~\ref{Lemma: Bound twin-width by twin-width of blocks, torso version}, we get the following
two asymptotic bounds on the twin-width of a graph admitting a tree decomposition of small adhesion.
{
\renewcommand{\thetheorem}{\ref*{Theorem: Bound twin-width by twin-width of blocks, hat + torso version}}
\addtocounter{theorem}{-1}%
\begin{theorem}
	For every \(k \in \mathbb{N}\) there exist explicit constants \(D_k\) and \(D_k'\) such that
	for every graph \(G\) with a tree decomposition of adhesion \(k\) and parts \(P_1,P_2,\dots,P_\ell\),
	the following statements are satisfied:
	\begin{enumerate}
	\item \(\displaystyle\tww(G)\leq 2^k\max_{i\in[\ell]}\tww(\widehat{P}_i)+D_k\),
	\item if \(k\geq 4\), then \(\displaystyle\tww(G)\leq \frac{2^k}{(k-1)!}\max_{i\in[\ell]} \tww(\overline{P}_i)^{k-1}+D_k'\).
	\end{enumerate}
\end{theorem}
}

\subsection{Tri- and quasi-4-connected components}
We now want to apply these general results on the interplay between twin-width and tree decompositions of small adhesion to obtain bounds on the twin-width of graphs in terms of the twin-width of their tri- and quasi-4-connected components.

{
\renewcommand{\thetheorem}{\ref*{Theorem: Bound on twin-width by twin-width of triconnected components}}
\addtocounter{theorem}{-1}%
\begin{theorem}
	Let \(G\) be a \(2\)-connected graph and let $C_1, C_2, \dots, C_{\ell}$ be its triconnected components.
	If we write \(\overline{C}_i\) for the red torsos of the triconnected components \(C_i\), then
	\[\tww(G)\leq \max\left(8\max_{i\in [\ell]}\tww(\overline{C}_i)+6,18\right).\]
	\begin{proof}
		This follows from Lemma~\ref{Lemma: Bound twin-width by twin-width of blocks, torso version}
		applied to the tree of triconnected components of~\(G\) together with the observation
		that for \(k=2\), the second term in the maximum in Lemma~\ref{Lemma: Bound twin-width by twin-width of blocks, torso version}
		is always bounded by the maximum of the first and third term.
	\end{proof}
\end{theorem}
}

Note that in Theorem~\ref{Theorem: Bound on twin-width by twin-width of triconnected components} we cannot hope for a lower bound similar to
the lower bound in Theorem~\ref{Theorem: Bound on twin-width by twin-width of biconnected components} without dropping the virtual edges.
Indeed, consider a \(3\)-connected graph~\(G\) of large twin-width (e.g. a Paley graph or a Rook's graph).
By \cite{twwleq4NPharda}, a \((2\lceil\log(|G|)\rceil-1)\)-subdivision
\(H\) of \(G\) has twin-width at most \(4\), but its triconnected components are \(G\) and multiple long cycles.
Thus, there exist graphs of twin-width at most \(4\) with triconnected components of arbitrarily large twin-width.

Moreover, the red virtual edges in each separator can also not be replaced by black edges.
Indeed, consider the class of \(1\)-subdivisions of cliques, which has unbounded twin-width
by \cite{tww2}. However, their triconnected components are just the clique itself and a triangle for every edge,
all of which have twin-width \(0\).

In the case of separators of size \(3\), we get two bounds on the twin-width of a graph in terms of its
quasi-4-connected components: one linear bound in terms of the subgraphs induced on the quasi-4-connected components
together with a common red neighbor for every \(3\)-separator along which the graph was split,
and one quadratic bound in terms of the red torsos of the quasi-4-connected components.
{
\renewcommand{\thetheorem}{\ref*{Theorem: Bound on twin-width by twin-width of quasi-4-connected components}}
\addtocounter{theorem}{-1}%
\begin{theorem}
	Let $C_1, C_2, \dots, C_{\ell}$ be the quasi-4-connected components of a 3-connected graph~\(G\).
	\begin{enumerate}
		\item For $i \in [\ell]$ we construct a trigraph $\widehat{C}_i$ by adding for every 3-separator \(S\) in $C_i$
		along which~$G$ was split a vertex \(v_S\) which we connect via red edges to all vertices in \(S\).
		Then
		\[\tww(G)\leq \max\left(8\max_{i\in [\ell]}\tww(\widehat{C}_i)+14, 70\right).\]
		\item For $i \in [\ell]$, denote by \(\overline{C}_i\) the \emph{red torso} of the quasi-4-connected component \(C_i\).
		Then
		\[\tww(G)\leq \max\left(4\max_{i\in[\ell]}\left(\tww(\overline{C}_i)^2+\tww(\overline{C}_i)\right)+14,70\right).\]
	\end{enumerate}
	\begin{proof}
	The two claims follow from Lemma~\ref{Lemma: Bound twin-width by twin-width of blocks, simpler gadgets}
	and Lemma~\ref{Lemma: Bound twin-width by twin-width of blocks, torso version} applied to the tree
	of quasi-4-connected components of \(G\)~\cite{Gro16a} together with the observation
	that also for \(k=3\), the second term in the maximum in Lemma~\ref{Lemma: Bound twin-width by twin-width of blocks, torso version}
	is always bounded by the maximum of the first and third term.
	\end{proof}
\end{theorem}
}

\section{Conclusion and further research}\label{sec:future:work}
We proved that $\tww(G)\leq \frac{3}{2}k + 1 + \frac{1}{2}(\sqrt{k\ln k} + \sqrt{k} + 2\ln k)$ if $G$ is a graph of strong tree-width at most $k$~(Theorem~\ref{thm: tww-vs-stw}). Moreover, we demonstrated that the twin-width of a Paley graph agrees with its strong tree-width (Corollary~\ref{cor: paley}).

We provided a detailed analysis of the relation between the twin-width of a graph and the twin-width of its highly connected components.
We proved a tight linear upper bound on the twin-width of a graph given the twin-width of its biconnected components (Theorem~\ref{Theorem: Bound on twin-width by twin-width of biconnected components}).
There is a linear upper bound for a slightly modified version of triconnected components (Theorem~\ref{Theorem: Bound on twin-width by twin-width of triconnected components}).
By further providing a quadratic upper bound on the twin-width of a graph given the twin-widths of its modified quasi-4-connected components (Theorem~\ref{Theorem: Bound on twin-width by twin-width of quasi-4-connected components}) we took one important step further to complete the picture
of the interplay of the twin-width of a graph with the twin-width of its highly connected components.
As a natural generalization of the above decompositions we considered graphs allowing for a tree decomposition of small adhesion (Theorem~\ref{Theorem: Bound twin-width by twin-width of blocks, hat + torso version} and Theorem~\ref{thm:  widthadhesion}).

It seems worthwhile to integrate our new bounds for practical twin-width computations, for example, with a branch-and-bound approach.

\bibliographystyle{plainurl}
\bibliography{bibliography.bib}

\end{document}